\documentclass[12pt]{amsart}
\usepackage{amssymb,float}
\usepackage[colorlinks=true]{hyperref}
\usepackage[all]{xy}
\usepackage[toc,page,title,titletoc,header]{appendix}
\usepackage{xcolor}
\usepackage{tikz-cd}
\usepackage{graphicx}
\usepackage{epigraph}

\begin{document}
	\pdfoutput=1
	\theoremstyle{plain}
	\newtheorem{thm}{Theorem}[section]
	\newtheorem*{thm1}{Theorem 1}
	
	\newtheorem*{thmM}{Main Theorem}
	\newtheorem*{thmA}{Theorem A}
	\newtheorem*{thm2}{Theorem 2}
	\newtheorem{lemma}[thm]{Lemma}
	\newtheorem{lem}[thm]{Lemma}
	\newtheorem{cor}[thm]{Corollary}
	\newtheorem{pro}[thm]{Proposition}
	\newtheorem{prop}[thm]{Proposition}
	\newtheorem{variant}[thm]{Variant}
	\theoremstyle{definition}
	\newtheorem{notations}[thm]{Notations}
	\newtheorem{rem}[thm]{Remark}
	\newtheorem{rmk}[thm]{Remark}
	\newtheorem{rmks}[thm]{Remarks}
	\newtheorem{defi}[thm]{Definition}
	\newtheorem{exe}[thm]{Example}
	\newtheorem{claim}[thm]{Claim}
	\newtheorem{ass}[thm]{Assumption}
	\newtheorem{prodefi}[thm]{Proposition-Definition}
	\newtheorem{que}[thm]{Question}
	\newtheorem{con}[thm]{Conjecture}
	
	\newtheorem{exa}[thm]{Example}
	\newtheorem*{assa}{Assumption A}
	\newtheorem*{algstate}{Algebraic form of Theorem \ref{thmstattrainv}}
	
	\newtheorem*{dmlcon}{Dynamical Mordell-Lang Conjecture}
	\newtheorem*{condml}{Dynamical Mordell-Lang Conjecture}
	\newtheorem*{congb}{Geometric Bogomolov Conjecture}
	\newtheorem*{congdaocurve}{Dynamical Andr\'e-Oort Conjecture for curves}
	
	\newtheorem*{pdd}{P(d)}
	\newtheorem*{bfd}{BF(d)}

	\newtheorem*{probreal}{Realization problems}
	\numberwithin{equation}{section}
	\newcounter{elno}                
	\def\points{\list
		{\hss\llap{\upshape{(\roman{elno})}}}{\usecounter{elno}}}
	\let\endpoints=\endlist
	\newcommand{\SH}{\rm SH}
	\newcommand{\Cov}{\rm Cov}
	\newcommand{\Tan}{\rm Tan}
	\newcommand{\res}{\rm res}
	\newcommand{\Om}{\Omega}
	\newcommand{\om}{\omega}
	\newcommand{\La}{\Lambda}
	\newcommand{\la}{\lambda}
	\newcommand{\mc}{\mathcal}
	\newcommand{\mb}{\mathbb}
	\newcommand{\surj}{\twoheadrightarrow}
	\newcommand{\inj}{\hookrightarrow}
	\newcommand{\zar}{{\rm zar}}
	\newcommand{\Exc}{{\rm Exc}}
		\newcommand{\Mod}{{\rm Mod}}
	\newcommand{\an}{{\rm an}}
	\newcommand{\red}{{\rm red}}
	\newcommand{\codim}{{\rm codim}}
	\newcommand{\Supp}{{\rm Supp\;}}
	\newcommand{\Leb}{{\rm Leb}}
	\newcommand{\rank}{{\rm rank}}
	\newcommand{\Ker}{{\rm Ker \ }}
	\newcommand{\Pic}{{\rm Pic}}
	\newcommand{\Der}{{\rm Der}}
	\newcommand{\Div}{{\rm Div}}
	\newcommand{\Hom}{{\rm Hom}}
	\newcommand{\Corr}{{\rm Corr}}
	\newcommand{\im}{{\rm im}}
	\newcommand{\Spec}{{\rm Spec \,}}
	\newcommand{\Nef}{{\rm Nef \,}}
	\newcommand{\Frac}{{\rm Frac \,}}
	\newcommand{\Sing}{{\rm Sing}}
	\newcommand{\sing}{{\rm sing}}
	\newcommand{\reg}{{\rm reg}}
	\newcommand{\Char}{{\rm char\,}}
	\newcommand{\Tr}{{\rm Tr}}
	\newcommand{\ord}{{\rm ord}}
	\newcommand{\bif}{{\rm bif}}
	\newcommand{\AS}{{\rm AS}}
	\newcommand{\FS}{{\rm FS}}
	\newcommand{\CE}{{\rm CE}}
	\newcommand{\PCE}{{\rm PCE}}
	\newcommand{\WR}{{\rm WR}}
	\newcommand{\PR}{{\rm PR}}
	\newcommand{\TCE}{{\rm TCE}}
	\newcommand{\diam}{{\rm diam\,}}

	\newcommand{\dist}{{\rm dist\,}}
	\newcommand{\id}{{\rm id}}
	\newcommand{\NE}{{\rm NE}}
	\newcommand{\Gal}{{\rm Gal}}
	\newcommand{\Min}{{\rm Min \ }}
	\newcommand{\Hol}{{\rm Hol \ }}
		\newcommand{\Rat}{{\rm Rat}}
			\newcommand{\FL}{{\rm FL}}
			\newcommand{\fm}{{\rm fm}}
	
	\newcommand{\Max}{{\rm Max \ }}
	\newcommand{\Alb}{{\rm Alb}\,}
	\newcommand{\Aff}{{\rm Aff}\,}
	\newcommand{\GL}{{\rm GL}\,}        
	\newcommand{\PGL}{{\rm PGL}\,}
	\newcommand{\Bir}{{\rm Bir}}
	\newcommand{\Aut}{{\rm Aut}}
	\newcommand{\End}{{\rm End}}
	\newcommand{\Per}{{\rm Per}\,}
	\newcommand{\Preper}{{\rm Preper}\,}
	\newcommand{\ie}{{\it i.e.\/},\ }
	\newcommand{\niso}{\not\cong}
	\newcommand{\nin}{\not\in}
	\newcommand{\soplus}[1]{\stackrel{#1}{\oplus}}
	\newcommand{\by}[1]{\stackrel{#1}{\rightarrow}}
	\newcommand{\longby}[1]{\stackrel{#1}{\longrightarrow}}
	\newcommand{\vlongby}[1]{\stackrel{#1}{\mbox{\large{$\longrightarrow$}}}}
	\newcommand{\ldownarrow}{\mbox{\Large{\Large{$\downarrow$}}}}
	\newcommand{\lsearrow}{\mbox{\Large{$\searrow$}}}
	\renewcommand{\d}{\stackrel{\mbox{\scriptsize{$\bullet$}}}{}}
	\newcommand{\dlog}{{\rm dlog}\,}    
	\newcommand{\longto}{\longrightarrow}
	\newcommand{\vlongto}{\mbox{{\Large{$\longto$}}}}
	\newcommand{\limdir}[1]{{\displaystyle{\mathop{\rm lim}_{\buildrel\longrightarrow\over{#1}}}}\,}
	\newcommand{\liminv}[1]{{\displaystyle{\mathop{\rm lim}_{\buildrel\longleftarrow\over{#1}}}}\,}
	\newcommand{\norm}[1]{\mbox{$\parallel{#1}\parallel$}}
	\newcommand{\boxtensor}{{\Box\kern-9.03pt\raise1.42pt\hbox{$\times$}}}
	\newcommand{\into}{\hookrightarrow}
	\newcommand{\image}{{\rm image}\,}
	\newcommand{\Lie}{{\rm Lie}\,}      
	\newcommand{\CM}{\rm CM}
		\newcommand{\Teich}{\rm Teich\;}
	\newcommand{\sext}{\mbox{${\mathcal E}xt\,$}}  
	\newcommand{\shom}{\mbox{${\mathcal H}om\,$}}  
	\newcommand{\coker}{{\rm coker}\,}  
	\newcommand{\sm}{{\rm sm}}
	\newcommand{\pgcd}{\text{pgcd}}
	\newcommand{\trd}{\text{tr.d.}}
	\newcommand{\tensor}{\otimes}
	\newcommand{\hotimes}{\hat{\otimes}}
	
	\newcommand{\CH}{{\rm CH}}
	\newcommand{\tr}{{\rm tr}}
	\newcommand{\e}{\rm SH}
	
	\renewcommand{\iff}{\mbox{ $\Longleftrightarrow$ }}
	\newcommand{\supp}{{\rm supp}\,}
	\newcommand{\ext}[1]{\stackrel{#1}{\wedge}}
	\newcommand{\onto}{\mbox{$\,\>>>\hspace{-.5cm}\to\hspace{.15cm}$}}
	\newcommand{\propsubset}
	{\mbox{$\textstyle{
				\subseteq_{\kern-5pt\raise-1pt\hbox{\mbox{\tiny{$/$}}}}}$}}
	\newcommand{\sA}{{\mathcal A}}
	\newcommand{\sB}{{\mathcal B}}
	\newcommand{\sC}{{\mathcal C}}
	\newcommand{\sD}{{\mathcal D}}
	\newcommand{\sE}{{\mathcal E}}
	\newcommand{\sF}{{\mathcal F}}
	\newcommand{\sG}{{\mathcal G}}
	\newcommand{\sH}{{\mathcal H}}
	\newcommand{\sI}{{\mathcal I}}
	\newcommand{\sJ}{{\mathcal J}}
	\newcommand{\sK}{{\mathcal K}}
	\newcommand{\sL}{{\mathcal L}}
	\newcommand{\sM}{{\mathcal M}}
	\newcommand{\sN}{{\mathcal N}}
	\newcommand{\sO}{{\mathcal O}}
	\newcommand{\sP}{{\mathcal P}}
	\newcommand{\sQ}{{\mathcal Q}}
	\newcommand{\sR}{{\mathcal R}}
	\newcommand{\sS}{{\mathcal S}}
	\newcommand{\sT}{{\mathcal T}}
	\newcommand{\sU}{{\mathcal U}}
	\newcommand{\sV}{{\mathcal V}}
	\newcommand{\sW}{{\mathcal W}}
	\newcommand{\sX}{{\mathcal X}}
	\newcommand{\sY}{{\mathcal Y}}
	\newcommand{\sZ}{{\mathcal Z}}
	\newcommand{\A}{{\mathbb A}}
	\newcommand{\B}{{\mathbb B}}
	\newcommand{\C}{{\mathbb C}}
	\newcommand{\D}{{\mathbb D}}
	\newcommand{\E}{{\mathbb E}}
	\newcommand{\F}{{\mathbb F}}
	\newcommand{\G}{{\mathbb G}}
	\newcommand{\HH}{{\mathbb H}}
	\newcommand{\LL}{{\mathbb L}}
	\newcommand{\J}{{\mathbb J}}
	\newcommand{\M}{{\mathbb M}}
	\newcommand{\N}{{\mathbb N}}
	\renewcommand{\P}{{\mathbb P}}
	\newcommand{\Q}{{\mathbb Q}}
	\newcommand{\R}{{\mathbb R}}
	\newcommand{\T}{{\mathbb T}}
	\newcommand{\U}{{\mathbb U}}
	\newcommand{\V}{{\mathbb V}}
	\newcommand{\W}{{\mathbb W}}
	\newcommand{\X}{{\mathbb X}}
	\newcommand{\Y}{{\mathbb Y}}
	\newcommand{\Z}{{\mathbb Z}}

	\newcommand{\bk}{{\mathbf{k}}}
	
	\newcommand{\bp}{{\mathbf{p}}}
	\newcommand{\ep}{\varepsilon}
	\newcommand{\bbk}{{\overline{\mathbf{k}}}}
	\newcommand{\Fix}{\mathrm{Fix}}
	
	\newcommand{\tor}{{\mathrm{tor}}}
	\renewcommand{\div}{{\mathrm{div}}}
	
	\newcommand{\trdeg}{{\mathrm{trdeg}}}
	\newcommand{\Stab}{{\mathrm{Stab}}}
	
	\newcommand{\OK}{{\overline{K}}}
	\newcommand{\ok}{{\overline{k}}}
	
	\newcommand{\cf}{{\color{red} [c.f. ?]}}
	\newcommand{\jy}{\color{red} jy:}

	\title[]{The multiplier spectrum morphism is generically injective}
	
	\author{Zhuchao Ji}
	
	\address{Institute for Theoretical Sciences, Westlake University, Hangzhou 310030, China}
	
	\email{jizhuchao@westlake.edu.cn}
	
	\author{Junyi Xie}

	
	\address{Beijing International Center for Mathematical Research, Peking University, Beijing 100871, China}
	
	\email{xiejunyi@bicmr.pku.edu.cn}

	
	\date{\today}

	\bibliographystyle{alpha}
	
	\maketitle
	
	\begin{abstract}
In this paper, we consider the multiplier spectrum of periodic points, which  is a natural morphism defined on the moduli space of rational maps on the projective line. A celebrated theorem of McMullen asserts that aside from the well-understood flexible Latt\`es family, the multiplier spectrum morphism is quasi-finite. In this paper, we strengthen McMullen's theorem by showing that the multiplier spectrum morphism is generically injective.  This answers a question of McMullen and Poonen. 	\end{abstract}

	\section{Introduction}
		\subsection{The multiplier spectrum morphism}
	For $d\geq 2,$ let $\Rat_d(\C)$ be the space of degree $d$ rational maps on $\P^1(\C)$.  
	It is a smooth irreducible affine variety of dimension $2d+1$ \cite{Silverman2012}.   A rational map  is called \emph{Latt\`es} if it is semi-conjugate to an endomorphism on an elliptic curve. A Latt\`es map $f$ is called \emph{flexible Latt\`es} if one can continuously vary the complex structure of the elliptic curve to get a family of Latt\`es maps passing through $f$. The structure of flexible Latt\`es maps is well-understood \cite[Lemma 5.5]{milnor2006lattes}.
	Let $\FL_d(\C)\subseteq \Rat_d(\C)$ be the locus of flexible Latt\`es maps, which is Zariski closed in $\Rat_d(\C)$.
	The group $\PGL_2(\C)= \Aut(\P^1(\C))$ acts on $\Rat_d(\C)$ by conjugacy. The geometric quotient 
	$$\sM_d(\C):=\Rat_d(\C)/\PGL_2(\C)$$ is the (coarse) \emph{moduli space} of endomorphisms  of degree $d$ \cite{Silverman2012}.
	The moduli space $\sM_d(\C)=\Spec (\sO(\Rat_d(\C)))^{\PGL_2(\C)}$ is an irreducible affine variety of dimension $2d-2$ \cite[Theorem 4.36(c)]{Silverman2007}, which is also a complex orbifold \cite{Milnor1993}, \cite{milnor2011dynamics}.
	Let $\Psi: \Rat_d(\C)\to \sM_d(\C)$ be the quotient morphism.  We set 
	$$\sM_d^{\star}(\C):=\sM_d(\C)\setminus \Psi(\FL_d(\C)).$$
We note that $ \FL_d(\C)=\emptyset$ when $d$ is not a square number,  and $ \Psi(\FL_d(\C))$ is an  algebraic curve when $d$ is a square number.
\par 
As $\sM_d$ is affine, to understand the geometry of $\sM_d$, we only need to understand its  ring of functions.
There is  a natural  dynamically interesting family of morphisms from $\sM_d(\C)$ to  affine spaces, which we call the \emph{multiplier spectrum morphisms}. 
These morphisms give a natural system of 
functions on $\sM_d(\C)$, which can be viewed as analogies of the theta functions on the moduli space of elliptic curves (c.f. \cite[Remark 5.9]{Silverman2023}).

We now recall the construction.  For every $f\in \Rat_d(\C)$ and $n\geq 1$, $f^n$ has exactly $N_n=d^n+1$ fixed points counted with multiplicity. 
The \emph{multiplier} of a $f^n$-fixed point $x$ is the differential $df^n(x)\in \C$. Using elementary symmetric polynomials, their multipliers define a point  $S_n(f)\in \C^{N_n}$. 
The \emph{multiplier spectrum} of $f$ is the sequence $S_n(f), n\geq 1.$  Since $S_n$ take the same value in a conjugacy class of rational maps, each $S_n$ defines a morphism $S_n:\sM_d(\C)\to \C^{N_n}$. Let $\tau_{d,n}$  be the morphism
\begin{align*}
\tau_{d,n}:\sM_d(\C)&\to \C^{N_1}\times\cdots\times \C^{N_n},\\
[f]&\mapsto (S_1(f),\dots,S_n(f)).
\end{align*}
Here for $f\in \Rat_d(\C)$, we define $\tau_{d,n}(f):=\tau_{d,n}([f])$.
\par For each $n\geq 1$, set $$R_n:=\{([f],[g])\in \sM_d(\C)^2|\,\,\tau_{d,n}([f])=\tau_{d,n}([g])\}.$$
Then  $R_n$  form a decreasing sequence of Zariski closed
subsets  of $ \sM_d(\C)^2$.
By the Noetherianity, the sequence $R_n$ is stable for $n$ sufficiently large. Hence there exists a minimal positive  integer $m_d$ such that  $\tau_{d,m_d}(f)=\tau_{d,m_d}(g)$ implies that  $\tau_{d,n}(f)=\tau_{d,n}(g)$ for every $n\geq 1$, i.e. $f$ and $g$ have the same multiplier spectrum.  We define $$\tau_d:= \tau_{d,m_d}.$$
\par It is well-known that elements in an irreducible component of $\FL_d(\C)$ have the same multiplier spectrum.   

The following remarkable theorem of McMullen \cite{McMullen1987} says that outside $\FL_d(\C)$, the multiplier spectrum determines the conjugacy class of rational maps up to finitely many choices.
\begin{thm}[McMullen]\label{McMullen}
For every $d\geq 2$, the morphism $$\tau_d:\sM_d^{\star}(\C)\to \C^{N_1}\times\cdots\times \C^{N_{m_d}}$$ is quasi-finite.
\end{thm}

	\subsection{The multiplier spectrum morphism is generically injective}
	
		\begin{defi}
		For a point $x\in\sM_d(\C)$, we say that $\tau_d$ is injective at $x$ if $\tau_d^{-1}(\tau_d(x))=\left\{x\right\}$.
		For  a subset $X\subset \sM_d(\C)$, we say that $\tau_d$ is injective on $X$ if $\tau_d$ is injective at every $x\in X$.
	\end{defi}
\par We quote the following question about the injectivity of $\tau_d$  from  McMullen \cite[Page 489]{McMullen1987}:
		\medskip
	\par {\em Noetherian properties imply there are an $N$ and an $M$ such that $E_1(R),\dots,E_N(R)$  determine $R$ up to at most $M$ choices, $\dots$ is $R$ determined uniquely?}\footnote{Following McMullen's notations, $E_i(R)$ is the multiplier spectrum of periods $i$ of the rational map $R$.}
	\medskip


	It turns out that $\tau_d$  is {\bf not} always  injective on $\sM_d^{\star}(\C)$.  Silverman showed that if $f$ is a  rigid Latt\`es map defined over a number field $K$ whose class number is larger than 1, then $\tau_d$ is not injective  at $[f]$ \cite[Theorem 6.62]{Silverman2007}. 
	
	Another construction of maps with the same multiplier spectrum was introduced by Pakovich \cite{Pakovich2019}. Let $f$ be a rational map. Following Pakovich \cite{Pakovich2019}, for any decomposition $f=h_1\circ h_2$ into a composition of rational maps, where $h_1$ and $h_2$ have degree at least 2, we say that the rational map $\tilde{f}:=h_2\circ h_1$ is an \emph{elementary transformation} of $f$.  We say that rational maps $f$ and $g$ are \emph{equivalent} if there exists a chain of elementary
	transformations between $f$ and $g$. One can show that if $f$ and $g$ are equivalent then $\tau_d(f)=\tau_d(g)$,  
	see Pakovich \cite[Lemma 2.1]{pakovich2019recomposing}.  
	
	\medskip
	
Even though  $\tau_d$  is not always injective on $\sM_d^{\star}(\C)$ as we have seen, one might hope that $\tau_d$ is  \emph{generically injective}, i.e. $\tau_d$ is injective on a Zariski open subset. 


 Poonen asked whether $\tau_d$ is always generically injective \cite[Question 2.43]{Silverman2012}.  The following is our main theorem. It  gives an affirmative answer to Poonen's question, hence an affirmative answer to McMullen's question for generic parameters. 
	
%
	
\begin{thm}\label{main}
		For every $d\geq 2$,
		 the morphism $$\tau_d:\sM_d(\C)\to \C^{N_1}\times\cdots\times \C^{N_{m_d}}$$ is generically injective. 
	\end{thm}
As a consequence, the system of functions given by the multiplier spectrum generate  the fraction field of $\sM_d.$  

One can also consider the moduli space of polynomials and the multiplier spectrum morphism on it.  Let $\sP_d$ be the moduli space of degree $d$ polynomials, which  is the quotient space of the space of degree $d$ polynomials modulo the conjugation by affine maps. For every $d\geq 2$ and $n\geq 1$, we let $\tilde{\tau}_{d,n}$ (resp. $\tilde{\tau}_{d}$) be the restriction of $\tau_{d,n}$ (resp. $\tau_{d}$) on $\sP_d$. As the the moduli space of polynomials is a proper Zariski closed subset of $\sM_d$, Theorem \ref{main} does not implies the injectivity of $\tilde{\tau}_{d}$ directly. However, our proof of Theorem \ref{main} still works in the polynomial setting. 
\begin{thm}\label{thmpoly}For every $d\geq 2$,  $\tilde{\tau}_{d}$ is generically injective on $\sP_d$. 
\end{thm}

While we were preparing this article, we knew from a private communication that Huguin has an independent proof of Theorem \ref{thmpoly}, see \cite{huguin2024moduli}.
Indeed Huguin proved the generic injectivity of $\tilde{\tau}_{d,2}$. 
From the communication with Huguin, 
we know that Huguin's method is completely different from ours, which is based on Fujimura's result \cite{Fujimura2007}  and on the computations  by Gorbovickis \cite{Gorbovickis2015}.
\subsection{Ingredients in the proof of Theorem \ref{main}}
To show Theorem \ref{main}, we argue by contradiction. Assume  that $\tau_d$ is not generically injective.  We first construct  two  algebraic families of rational maps $f_t$ and $g_t$, parametrized by the same algebraic curve $V$, such that: (1)  for every $t\in V(\C)$, the images $f_t$ and $g_t$ in $\sM_d$ are different; (2) $\tau_d(f_t)=\tau_d(g_t)$ for every $t\in V(\C)$. We can further ask that these two families satisfy several geometric and dynamical assumptions, c.f. Lemma \ref{curve}. Our key step is  to show that $f_t$ and $g_t$ are  \emph{intertwined} (c.f. Definition \ref{defiintert} ) for all but finitely many $t\in V(\C)$. 
The main ingredient of the proof for this step is the following result in the spirit of the Dynamical Andr\'e-Oort (DAO) conjecture.  The oringinal DAO conjecture was proposed by Baker-DeMarco \cite{baker2013special}. The DAO conjecture is about  the distribution of \emph{ postcritically finite} (PCF) maps in a subvariety of the moduli space. A rational map is called PCF if its critical orbits is a finite set.  The one-dimensional DAO conjecture concerning distribution of PCF maps in an algebraic curve was recently solved by the authors \cite{ji2023dao}.  Instead of considering  one algebraic family of rational maps, the following DAO-type result is about two algebraic family of rational maps. See Definition  \ref{family} for the definition of algebraic families.

\begin{thm}\label{DAOint}
Let $d\geq 2$ and let $f_V,g_V$ be two degree $d$ non-isotrivial algebraic families defined over $\C$, parametrized by the same irreducible algebraic curve $V$, and $f_V,g_V$  are not families of flexible Latt\`es maps. Assume that there are infinitely many $t\in V(\C)$ such that $f_t$ and $g_t$ are both PCF. Then for all but finitely many $t\in V(\C)$, $f_t$ and $g_t$ are intertwined. 
\end{thm}

\begin{rem}
In Theorem \ref{DAOint}, by working harder, we actually can show that for {\bf every} point $t\in  V(\C)$, $f_t$ and $g_t$ are intertwined. As we do not  need this stronger result for the proof of Theorem \ref{main}, we only sketch a proof here. Let $\eta$ be the generic point of $V.$
As $f_\eta$ and $g_\eta$ are intertwined, after replacing $f_V,g_V$ by a suitable iterate, there is an $f_\eta\times g_\eta$-invariant irreducible  subvariety $\Gamma\subseteq \P^1_\eta\times \P^1_{\eta}$ which dominants the two factors. After a base change by a finite morphism $V'\to V$, we may assume that $V$ is smooth and $\Gamma_{\eta}$ is geometrically irreducible. Set $h_{\eta}:=f_{\eta}\times g_{\eta}|_{\Gamma_{\eta}}$. We have $\deg h_{\eta}=d\geq 2.$
For every point $t\in V(\C)$, denote by $\Gamma_t\subseteq (\P^1_\C)^2$ the specialization of $\Gamma_{\eta}$ at the fiber above $t.$ As the numerical class of $\Gamma_t$ is constant in the family, we only need to show that $\Gamma_t$ is irreducible. Every point $t\in V(\C)$ induces an embedding of $\C(V)$ to the non-archimedean field $\C((T))$. This induces an endomorphism $h^{\an}$ on the Berkovich analytification  $\Gamma^{\an}$ of $\Gamma_\eta$ (which depends on the choice of $t$).  Each irreducible component of $\Gamma_t$ defines a type-II point in $\Gamma^{\an}$. These type-II points form a totally invariant finite subset of $h^{\an}$ in  $\Gamma^{\an}.$ It is well-known in one-dimensional non-archimedean dynamics that such a  totally invariant finite subset  must be a single point \cite[Proposition 8.13]{Benedetto2019a}. Hence $\Gamma_t$ is irreducible, which concludes the proof. This proof indeed shows that the intertwined locus in $\sM_d\times \sM_d$ is Zariski closed.
\end{rem}


  Before \cite{ji2023dao}, the one-dimensional DAO conjecture and related results have been studied in \cite{baker2013special}, \cite{ghioca2013preperiodic}, \cite{ghioca2015preperiodic}, \cite{de2015bifurcation}, \cite{ghioca2016unlikely}, \cite{ghioca2016case},  \cite{ghioca2017dynamical}, \cite{favre2018classification}, \cite{ghioca2018dynamical}, \cite{ghioca2018simultaneously}, \cite{favre2022arithmetic}.  In particular, Theorem \ref{DAOint} was proved by Favre-Gauthier \cite{favre2022arithmetic} in the case that $f_V$ and $g_V$ are polynomial families.
 
 \medskip
\par  From the above construction  we can get a Zariski dense set of \emph{simple} (c.f. Definition \ref{simple}) rational maps $[f]$ in $\sM_d(\C)$ 
which is intertwined with a simple rational map $g$ with $[g]\neq [f]$.   We then use a result of Pakovich (c.f. Theorem \ref{intertwined}) to get a contradiction.
\medskip


\medskip

In a private communication, DeMarco and Mavraki told us that when they preparing lectures in Harvard and Toronto in November 2023  \cite{DeMarco2023}, they find a simplification of our proof of  Theorem \ref{main}. They observed that to prove Theorem \ref{main}, one may replace Theorem \ref{DAOint} by a weaker lemma
(c.f. Lemma \ref{lemdmweak}). This lemma can be proved following the same strategy of Theorem \ref{DAOint}, but there are two important steps become easier. See Section  \ref{subsectionsimp} for details.

\subsection{Previous results on the degree of the multiplier spectrum morphism}
In \cite{Gorbovickis2015}, Gorbovickis showed that $\tau_{d,n}$ is generically quasi-finite for $d\geq 2$ and $n\geq 3.$
A recursive formula for the  upper bound of the topological degree of  $\tau_{d,n}$ was obtained by Schmitt in the preprint version \cite{Schmitt2016}.
An explicit upper bound of the topological degree of $\tau_{d,n}$ was obtained in Gotou's recent work \cite{Gotou2023}.

 When $d=2$, Milnor \cite{Milnor1993} showed that  $\tau_{2,1}$ is in fact injective on $\sM_d(\C)$ (see also \cite[Theorem 2.45]{Silverman2012}). In particular Theorem \ref{main} holds when $d=2$.
When $d=3$,  Gotou showed that $\tau_{3,2}$ is generically injective (but not injective) \cite[Theorem 1.2]{Gotou2023}.
This was mentioned in \cite{Hutz2013a} which is the  errata for \cite{Hutz2013}.  
Previously there was no result about the generic injectivity of $\tau_d$ when $d\geq 4$.

One can also consider the moduli space of polynomials and the multiplier spectrum morphism on it. For every $d\geq 2$ and $n\geq 1$, we let $\tilde{\tau}_{d,n}$ be the restriction of $\tau_{d,n}$ on the moduli space of polynomials. In this case $\tilde{\tau}_{d,1}$ is generically quasi-finite (while $\tau_{d,1}$ is never generically quasi-finite, except when $d=2$).  Fujimura showed that $\deg (\tilde{\tau}_{d,1})=(d-2)!$ \cite{Fujimura2007}. The fiber structure of $\tilde{\tau}_{d,1}$ was studied by Sugiyama \cite{Sugiyama2017}, \cite{Sugiyama2020}, \cite{Sugiyama2023}. 

\subsection{Further  results and problems}
\subsubsection{McMullen's conjecture for the hyperbolicity of the moduli space $\sM_f$} 

For a rational map $f\in \Rat_d(\C)$, McMullen and Sulllivan introduced the Teichm\"uller space $\sT_f$ and the moduli space $\sM_f$ \cite{mcmullen1998quasiconformal}.   We refer the readers to \cite{mcmullen1998quasiconformal} and \cite{astorg2017teichmuller} for the definitions of  $\sT_f$ and $\sM_f$.

\par In \cite[Page 473]{McMullen1987}, McMullen conjectured that for every $f\in \Rat_d(\C)$ that are not flexible Latt\`es, bounded holomorphic functions separate points on $\sM_f$.

In a forthcoming paper \cite{ji2023moduli}, we solve McMullen's conjecture  using the multiplier spectrum.  Moreover by applying Theorem \ref{main}, we get a more explicit description of $\sM_f$ when $d\geq 4$. 

\subsubsection{Injective locus of $\tau_d$}

We have seen before that there are two mechanisms that can  produce rational maps with the same multiplier spectrum, one from Latt\`es maps, and another one from equivalent rational maps. Pakovich asked  \cite[Problem 3.1]{Pakovich2019} whether 
these are the only obstructions of the injectivity of $\tau_d$.

In a forthcoming paper, using the tools developed in this paper, we can show that the non-isolated components of the non-injective locus of $\tau_{d}$ come from 
intertwined families.
This motivates us to conjecture that the answer of Pakovich's question is yes.

\begin{con}\label{conequmulti}
Let $f,g$ be rational maps of degree $d\geq 2$ such that the conjugacy classes of $f$ and $g$ are different. Assume that $\tau_d(f)=\tau_d(g)$, then one of the followings holds:
\begin{points}
\item $f$ and $g$ are Latt\`es maps;
\item $f$ is equivalent to $g$.
\end{points}
\end{con}


\subsubsection{Generic injectivity of multiplier spectrum of small periods}
In Theorem \ref{McMullen} and Theorem \ref{main}, the definition of $\tau_d$ requires the use of periodic points of periods not exceeding the number $m_d$, whose precise value is not effectively-known and is probably very large.  It is  interesting to know whether we can get generic injectivity using only periodic points of small periods.    By dimension counting, $\tau_{d,1}$ is never generically injective except when $d=2$.  Again by  dimension counting, it is very likely that  $\tau_{d,2}$ is generically injective.
\begin{con}
	For every $d\geq 2$,
the morphism $$\tau_{d,2}:\sM_d(\C)\to \C^{N_1}\times \C^{N_2}$$ is generically injective. 
\end{con}

\subsubsection{Generic injectivity of the length spectrum}
Replacing  the multipliers by their norm in the definition of multiplier spectrum, one gets the definition of the \emph{length spectrum}.
More precisely, for every $f\in \text{Rat}_d(\C)$ and $n\geq 1$, we denote by $L_n(f)\in \R_{\geq 0}^{N_n}$ the elements corresponding to the values of the elementary symmetric polynomials at the point $(|\la_1|,\dots, |\la_{N_n}|)\in\R_{\geq 0}^{N_n}$ where 
$\la_i, i=1,\dots, N_n$ are the multipliers of all $f^n$-fixed points. The length spectrum of $f$ is defined by the sequence $L_n(f), n\geq 1.$ A priori,  the length spectrum contains less information than the multiplier spectrum.  In \cite{ji2023homoclinic}, a parallel result of Theorem \ref{McMullen} has been shown, that is outside the flexible Latt\`es family,  the length spectrum determines the conjugacy class of rational maps up to finitely many choices \cite[Theorem 1.5]{ji2023homoclinic}. We believe the following parallel result of Theorem \ref{main} for length spectrum is true. Note that $\sM_d$ is defined over $\Q$ (hence over $\R$).

We propose the following conjecture as a parallel result of our Theorem \ref{main}.

\begin{con}\label{1.8}
	For every $d\geq 2$, there is a Zariski closed proper subset $E_d\subset \sM_d$ defined over $\R$, such that for every $[f]\notin E_d$, if there is $g\in\Rat_d(\C)$ such that $f$ and $g$ have the  same length spectrum, then  $g$ or $\overline{g}$ (the complex conjugation of $g$)  is conjugated to $f$.
\end{con}

More precisely, we propose the following description for rational maps with the same length spectrum.
\begin{con}\label{1.9}
Let $f,g$ be non-Latt\`es rational maps of degree $d\geq 2$. If $f$ and $g$ has the same length spectrum, then $\tau_d(f)$ equals to either $\tau_d(g)$ or $\tau_d(\overline{g})$.
\end{con}
Theorem \ref{main} and Conjecture \ref{1.9} implies 
Conjecture \ref{1.8} directly.


\subsubsection{Multiplier spectrum for endomorphisms on $\P^N$} 
For holomorphic endomorphisms on $\P^N$, $N\geq 2$, one can also construct the corresponding moduli space $\sM_d^N(\C)$ \cite{Silverman2012}, and the multiplier spectrum morphisms exist on  $\sM_d^N(\C)$ \cite{Silverman2012}. Unlike in dimension one, essentially nothing is known about the multiplier spectrum morphisms  when $N\geq 2$. It is of great interest to extend McMullen's quasi-finiteness theorem  (Theorem \ref{McMullen}) and the generic injectivity theorem (Theorem \ref{main}) to higher dimension.

%

\subsection*{Acknowledgement}
We thank Rin Gotou for sharing us his work on the generic injectivity of $\tau_{3,2}$ and let us know the result of Schmitt \cite{Schmitt2016}. We thank Fedor Pakovich, Valentin Huguin, Gabriel Vigny and Joseph Silverman for their comments and suggestions. We thanks Valentin Huguin for telling us his forthcoming proof for the generic injectivity of $\tilde{\tau}_{d,2}$. We thank Laura DeMarco and Niki Myrto Mavraki for sharing us their simplification of our proof of Theorem \ref{main}.

The first-named author would like to thank Beijing International Center for Mathematical Research in Peking University for the invitation.  Zhuchao Ji is supported
by ZPNSF grant (No.XHD24A0201) and NSFC Grant (No.12401106).  Junyi Xie is supported by NSFC Grant (No.12271007).
		
\section{Hyperbolic PCF maps of disjoint type}\label{curves}
\begin{defi}
A PCF map $f\in \Rat_d(\C)$ is called \emph{hyperbolic of disjoint type} if $f$ has  $2d-2$ distinct super-attracting cycles.  Let $\underline{n}:=\left\{n_1,\dots,n_{2d-2}\right\}\in (\N^*)^{2d-2}$, where $\N^*$ is the set of positive integers. A hyperbolic PCF map of disjoint type $f\in \Rat_{d}(\C)$ is called of type $\underline{n}$ if $f$ has $2d-2$ distinct super-attracting cycles of respective exact
periods $n_1,\dots,n_{2d-2}$. 
\end{defi}

 For every $\underline{n}:=\left\{n_1,\dots,n_{2d-2}\right\}\in (\N^*)^{2d-2}$, we let $X_{\underline{n}}\subset \sM_d(\C)$  be the subset of all conjugacy classes  $[f]$ such that $f$ is a hyperbolic PCF map of disjoint type $\underline{n}$.   The following  theorem was proved by Gauthier-Okuyama-Vigny in \cite[Theorem 6.6]{gauthier2019hyperbolic}.
\begin{thm}\label{GOV}

For every sequence   $\underline{n}(k)=(n_1(k),\dots, n_{2d-2}(k))\in (\N^*)^{2d-2}$
satisfying $\min_{1\leq j\leq 2d-2} n_j(k)\to +\infty$ when $k\to +\infty$,   we have that  $\cup_{k}X_{\underline{n}(k)}$ is Zariski dense in $\sM_d(\C)$. 
\end{thm}
\begin{defi}\label{family}
Let $V$ be a quasi-projective variety over $\C$.  An \emph{algebraic family} on $V$ is an endomorphism $f_{V}$ on $ V\times \P^1$ defined over $\C$ of the following form.
\begin{align*}
	f_V:V\times \P^1&\to V \times \P^1 \\
	(t,z)&\mapsto (t,f_t(z)).
\end{align*} 
We say that $f_V$ has degree $d$ if for every $t\in V(\C)$, we have $\deg f_t=d.$ For a degree $d$ algebraic family $f_{V}$ on $V$, let $\Psi_{V}: V\to \sM_d(\C)$ be the  morphism sending $t\in V(\C)$ to the class of $f_t$ in $\sM_d(\C).$	We say that $f_{V}$ is \emph{isotrivial} if $\Psi_{V}: V
\to \sM_d(\C)$ is locally constant. 
\end{defi}

 Let $V$ be an irreducible quasi-projective curve over $\C$.   Let $\pi_1:V\times\P^1\to V$ and $\pi_2:V\times\P^1\to \P^1$ be the canonical projections. Let $\omega_{\P^1}$ be the Fubini-Study form on $\P^1(\C)$. Let   $\omega_2:=\pi_2^\ast (\omega_{\P^1})$.  The {\em relative Green current} of $f$ is defined by 
\begin{equation*}
	T_f:=\lim_{n\to+\infty} d^{-n} (f^n)^\ast (\omega_2)=\omega_2+dd^c g,
\end{equation*}
where $g$ is a H\"older continuous quasi-p.s.h. function \cite[Lemma 1.19]{dinh2010dynamics}.
\par For every $t_0\in V(\C)$, we have $T_f\wedge [t=t_0]=\mu_{f_{t_0}}$, where $\mu_{f_{t_0}}$ is the {\em maximal entropy measure} of $f_{t_0}$.
\par A {\em marked point} $a$ is a morphism $a:V\to \P^1$.  The {\em bifurcation measure} of the pair $(f,a)$ is defined by
\begin{equation*}
	\mu_{f,a}:=(\pi_1)_\ast (T_f\wedge [\Gamma_a])=a^*T_f,
\end{equation*}
where $\Gamma_a$ is the graph of $a$, and $[\Gamma_a]$ is the current of integration on $\Gamma_a(\C)$.

\medskip
The following lemma is  a combination of \cite[Proposition 2.4]{dujardin2008distribution} and \cite[Theorem 1.1]{demarco2016bifurcations}.
\begin{lem}\label{bifur}
Let  $f_V:V\times \P^1\to V\times \P^1$  be a degree $d$  non-isotrivial algebraic family, and let $a:V\to \P^1$ be a marked point. Assume that $a$ is not persistently preperiodic,  then there exist infinitely many $t\in V(\C)$ such that $a(t)$ is periodic for $f_t$.
\end{lem}
\begin{proof}
By  \cite[Theorem 1.1]{demarco2016bifurcations}, the bifurcation locus $\text{Bif}(f_V,a)$ is a non-empty closed set.  Moreover $\text{Bif}(f_V,a)$ is the support of a positive closed current with H\"older continuous potential \cite[Lemma 1.1]{dinh2010dynamics}, hence $\text{Bif}(f_V,a)$  has positive Hausdorff dimension \cite[Theorem 1.7.3]{sibony1999dynamique}, in  particular  $\text{Bif}(f_V,a)$ is an infinite set. By definition, for every $t_0\in \text{Bif}(f_V,a)$ and every open neighborhood $U$ of $t_0$, the family of maps $h_n:U\to \P^1$, $t\mapsto f^n_t(a(t))$ does not form a normal family.  By \cite[Proposition 2.4]{dujardin2008distribution}, there exists $t\in U$ such that $a(t)$ is periodic for $f_t$. Apply the above construction to infinitely many disjoint open subsets $U$ meeting $\text{Bif}(f_V,a)$, we get  infinitely many $t\in V(\C)$ such that $a(t)$ is periodic for $f_t$.
\end{proof}

		For every $\underline{n}:=\left\{n_1,\dots,n_{2d-3}\right\}\in (\N^*)^{2d-3}$ of $2d-3$ tuples,  let $Y_{\underline{n}}\subset \sM_d(\C)$  be the subset of  all conjugacy classes  $[f]$ such that $f$ has $2d-3$ distinct super-attracting cycles of respective exact
	periods $n_1,\dots,n_{2d-3}$.

\begin{lem}\label{curve}
The following statements are  true:
\medskip
\par (1) The Zariski closure of $Y_{\underline{n}}$ is a (possibly reducible) algebraic curve provided that it is not empty and $Y_{\underline{n}}$ is Zariski open in its Zariski closure.
\medskip
\par (2) The set $\cup_{\underline{n}}Y_{\underline{n}}$ is Zariski dense in $\sM_d(\C)$. 
\medskip
\par (3) For every $\underline{n}$ such that  $Y_{\underline{n}}$ is non-empty, and for every irreducible component $Y$ of $Y_{\underline{n}}$, there are infinitely  many $[f]\in Y$ such that $f$ is a hyperbolic PCF map of disjoint type.
\end{lem}
\begin{proof}
We first show (1). Passing to a finite  morphism $\phi: V\to \Rat_d(\C)$, $t\mapsto f_t$, we can choose an algebraic family $f_V:V\times \P^1\to V\times \P^1$ such that all $2d-2$ critical points in this family can be algebraically parametrized by $c_1,\dots,c_{2d-2}:V\to\P^1$.  	For every fixed $\underline{n}:=\left\{n_1,\dots,n_{2d-3}\right\}\in (\N^*)^{2d-3}$, we define $$Z:=\left\{t\in V(\C): f_t^{n_j}(c_j)=c_j, \;\;\text{for every}\;1\leq j\leq 2d-3\right\}.$$ 
Then $Z$ is a Zariski closed subset in $V$. Since $\phi:V\to \Rat_d(\C)$ is surjective finite,  and $\phi(Z)\subset \Rat_d(\C) $ is invariant by $\PGL_2(\C)$ conjugacy,  $\Psi\circ\phi(Z)$ is Zariski closed in $\sM_d(\C)$, where $\Psi:\Rat_d(\C)\to \sM_d(\C)$ is the quotient morphism.  Moreover by the definition of $Y_{\underline{n}}$,  $Y_{\underline{n}}$ is Zariski open in $\Psi\circ\phi(Z)$. For each irreducible component $\tilde{Z}$ of $Z$, it is the intersection of at most $2d-3$ hypersurfaces, so $\tilde{Z}$ has codimension at most $2d-3$ in $V$.  We claim that  $\tilde{Z}$ has codimension $2d-3$ in $V$. Assume by contradiction it is not the case. Since  critical orbit relations are constant along the fibers of the projection $\Psi\circ\phi:V\to \sM_d(\C)$, $\Psi\circ\phi(\tilde{Z})$ has codimension $k<2d-3$ in $\sM_d(\C)$.  In particular the algebraic family $f_{\tilde{Z}}$ (the restriction of $f_V$ on $\tilde{Z}$) is non-isotrivial.  Apply Lemma \ref{bifur} to $f_{\tilde{Z}}$ and the marked point $c_{2d-2}:V\to \P^1$, there exists a positive integer $m$ such that the Zariski closed subset 
$$W:=\left\{t\in \tilde{Z}:f_t^m(c_{2d-2})=c_{2d-2}\right\}$$
is non-empty.   Moreover, $W$ has codimension at most $k+1<2d-2$.  Then $\Psi\circ\phi(W)$ has codimension at most $k+1<2d-2$, which implies that  $\Psi\circ\phi(W)$ contains a positive dimensional hyperbolic PCF family. This contradicts Thurston's rigidity theorem \cite{Douady1993} which says that the only positive dimensional PCF family in $\sM_d(\C)$ is the flexible Latt\`es family.  Hence $\tilde{Z}$ must  have codimension $2d-3$ in $V$. This implies $\Psi\circ\phi(Z)$ has pure dimension 1. Then (1) holds since $Y_{\underline{n}}$ is  Zariski open in $\Psi\circ\phi(Z)$.

\medskip

\par Next we show (2).  Since  $\cup_{\underline{n}}X_{\underline{n}}\subset \cup_{\underline{n}}Y_{\underline{n}}$, by Lemma \ref{GOV} (2), $\cup_{\underline{n}}Y_{\underline{n}}$ is Zariski dense.

\medskip

\par Finally we show (3). By Lemma \ref{curve} (1) and by Lemma \ref{bifur}, there exist infinitely many $[f]$ in 
$Y_{\underline{n}}$ such that every citical point of $f$ is periodic and  $f$ has $2d-3$ distinct super-attracting cycles of respective exact periods $n_1,\dots,n_{2d-3}$.  Since the set of conjugacy class $[f]$ such that all critical points of $f$ are periodic with bounded periods is finite, there exist infinitely many $[f]$ in 
$Y_{\underline{n}}$ such that $f$ has a periodic critical point with exact periods $N>\max_{1\leq j\leq 2d-3} n_j$, hence these $[f]$ are hyperbolic PCF maps of disjoint type.  This implies (3).
\end{proof}

\section{Generic injectivity of the multiplier spectrum morphism}
\subsection{Intertwined Rational maps}
\begin{defi}\label{defiintert}
Let $d\geq 2$ and $f,g\in\Rat_d(\C)$. We say $f$ and $g$ are \emph{intertwined} if there exists a (maybe reducible) algebraic curve $Z\subset \P^1\times \P^1$ whose projections to both axis are  surjective,
and $Z$ is invariant by the map $f\times g:\P^1\times \P^1\to \P^1\times \P^1$.  
\end{defi}
\begin{defi}\label{simple}
A rational map $f\in\Rat_d(\C)$ is called \emph{simple} if the $f$ has exactly $2d-2$ critical values.  
\end{defi}
Simple rational maps form a Zariski open subset of  $\Rat_d(\C)$ for every $d\geq 2$. 

\medskip

The following theorem is an easy corollary of results of  Pakovich \cite{pakovich2021iterates},  \cite{pakovich2021iterates}. We give a proof here for completeness.
\begin{thm}\label{intertwined}
Let $d\geq 4$. Then there exists a Zariski open subset $W_d$ of $\Rat_d(\C)$ such that :
\begin{points}
	\item Elements in $W_d$ are simple;
	\item Assume that $f$ and $g$ are intertwined, where $f\in W_d$ and $g$ is simple, then $f$ and $g$ are in the same conjugacy class.
\end{points}
\end{thm}
\begin{proof}
 Let $d\geq 4$, by \cite[Theorem 1.2 and Lemma 3.7]{pakovich2021iterates}, there exists a Zariski open set $W_d\subset \Rat_d(\C)$ such that: (1) $W_d$ is contained in the set of simple rational maps, (2) if $f\in W_d$ and $f$ share the maximal entropy measure with another rational map $g\in \Rat_d(\C)$, then $g=f$.  We are going to show that the set $W_d$  satisfies Theorem \ref{intertwined} (ii).  Assume that $f\in W_d$, and $f$ is  intertwined with a simple rational map $g\in\Rat_d(\C)$.  By \cite[Theorem 1.4]{pakovich2021iterates},  there exist $m\geq 1$ and $\phi\in\PGL_2(\C)$ such that $(\phi g\phi^{-1})^m=f^m$.  This implies $f$ and $\phi g\phi^{-1}$ have the same maximal entropy measure, hence we have $f=\phi g\phi^{-1}$.  Hence $f$ and $g$ are in the same conjugacy class.
\end{proof}
\subsection{A DAO-type theorem}

We first introduce some terminology in complex dynamics, which are used in the proof of Theorem \ref{DAOint}. See \cite{ji2023dao} for a reference.

\medskip

{\bf TCE}: A rational map $g:\P^1(\C)\to \P^1(\C)$ is called $\TCE(\la)$ for some $\la>1$ if  there exists $\delta_0>0$ such that for each $n\geq 0$ and $z\in \sJ(g)$, we have 
$$\diam W_n\leq \la^{-n},$$
where $W_n$ is any connected component of $g^{-n}(B(z,\delta_0))$, and $\sJ(g)$ is the Julia set of $g$.

\medskip
{\bf CE}: A rational map $g:\P^1(\C)\to \P^1(\C)$ is called $\CE(\la)$ for some $\la>1$, if there exists $C>0$ and  $N>0$ such that  for every critical points $c\in \sJ_g$ and $n\geq 0$, $$|dg^n(g^N(c))|\geq C\la^n.$$  Here $dg^n$ is the derivative of $g^n$. Moreover $g$ does not have parabolic cycles.
\medskip

Let $f_V$ be an algebraic family of  degree $d$ rational maps parametrized by a smooth quasiprojective curve $V$ over $\C$. Let $c_1,\dots, c_{2d-2}$ the marked critical points.  For every marked critical point $c$ and $n\geq 0$, let $\xi_{c,n}:V(\C)\to \P^1(\C)$ denote the map $\xi_{c,n}(t):=f_t^n(c(t))$. 

\medskip

{\bf PCE}:  A parameter $t_0\in V(\C)$ is called $\PCE(\la)$ for some $\la>1$, if there exists $C>0$ such that for every marked critical point $c$ and $n\geq 0$,  $$\left|\frac{d\xi_{c,n}}{dt}(t_0)\right|\geq C\la^n.$$

\medskip

{\bf PR}:  A parameter $t_0\in V(\C)$ is called $\PR(s)$ for some $s>0$, if there exists an integer $N>0$ such that for every marked critical point $c$,  
$$ \dist(f_{t_0}^n(c(t_0)), \mathcal{C}_{t_0})\geq n^{-s}$$
for every $n\geq N$.  Where $\mathcal{C}_{t_0}$ is the critical set of $f_{t_0}$  and $\dist(\cdot, \cdot)$ is the distance function.

\medskip

The following theorem is similar to the main theorem in \cite{ji2023dao}. In \cite[Theorem 1.2]{ji2023dao}, the proof is for the product map $f_V\times_V  f_V$ on $\P^1\times \P^1$, here we replace this product map by  $f_V\times_V g_V$ on $\P^1\times \P^1$.  Indeed the proof of Theorem \ref{DAO} is easier since we only need to construct a dynamical relation between $f_V$ and $g_V$.

\begin{thm}[=Theorem \ref{DAOint}]\label{DAO}
Let $d\geq 2$ and let $f_V,g_V$ be two degree $d$ non-isotrivial algebraic families parametrized by the same irreducible algebraic curve $V$. Assume $f_V$ and $g_V$  are not  flexible Latt\`es family. Assume that there are infinitely many $t\in V(\C)$ such that $f_t$ and $g_t$ are both PCF. Then for all but finitely many $t\in V(\C)$, $f_t$ and $g_t$ are intertwined. 
\end{thm}

\proof 
Our proof is a modification of the one for \cite[Theorem 1.3]{ji2023dao}.

We first reduce Theorem \ref{DAO} to the case such that $V, f_V, g_V$ are defined over $\overline{\Q}.$
To show this, let $\Psi_f,\Psi_g: V\to \sM_d$ be the morphisms sending $t$ to the conjugacy classes of $f_t$ and $g_t$ respectively.
Denote by $\psi: V\to \sM_d\times \sM_d$ the morphism sending $t$ to $(\Psi_f(t),\Psi_g(t))$ and by $\Gamma$ the Zariski closure of its image. 
Since there are infinitely many $t\in V(\C)$ such that $f_t$ and $g_t$ are both PCF and $f_V,g_V$  are not family of flexible Latt\`es maps, by Thurston's rigidity theorem for PCF maps \cite{Douady1993}, $\Gamma$ is defined over $\overline{\Q}.$
Consider the natural morphism $\Psi\times\Psi: \Rat_d\times \Rat_d\to \sM_d\times \sM_d$.
There is an algebraic curve $V'$ in $(\Psi\times\Psi)^{-1}(\Gamma)$ defined over $\overline{\Q}$ such that $\Psi\times\Psi(V')$ is dense in $\Gamma.$
Then $V'$ defines algebraic families $f'_{V'}, g'_{V'}$ of degree $d$ maps parametrized by $V'$. Both of them are defined over $\overline{\Q}.$
To prove Theorem \ref{DAO} for $f_V,g_V$, we only need to prove it for $f'_{V'}, g'_{V'}$. So we may assume now that $V, f_V, g_V$ are defined over $\overline{\Q}.$

After replacing $V$ by its normalization and then by some finite ramification cover, we may assume that $V$ is smooth and both $f_V$ and $g_V$ have exactly $2d-2$ marked critical points $a_1,\dots, a_{2d-2}$ and $b_1,\dots, b_{2d-2}.$

Now we follow the notations in \cite{ji2023dao}. We denote by $\mu_{f,a_i},\mu_{g,b_i}$ the bifucation measures on $V(\C)$ for pairs $(f_V, a_i)$ and $(g_V, b_i)$ respectively.
By \cite[Theorem 2.5]{dujardin2008distribution} (and DeMarco \cite{demarco2016bifurcations}), $\mu_{f_V,a_i}$  (resp. $\mu_{g_V,b_i}$) vanishes if and only if $a_i$ (resp. $b_i$)
is preperiodic. In this case, they are called \emph{passive}. Otherwise, they are called \emph{active}. Let $\mu_{f_V,\bif}:=\sum_{i=1}^{2d-2}\mu_{f_V,a_i}$ and $\mu_{g_V,\bif}:=\sum_{i=1}^{2d-2}\mu_{g_V,b_i}$ be the bifurcation measures for $f_V$ and $g_V$ respectively. By Thurston's rigidity theorem for PCF maps \cite{Douady1993}, both $\mu_{f_V,\bif}$
and $\mu_{g_V,\bif}$ are non-zero. We may assume that $a_1$ and $b_1$ are active. 

\medskip

By \cite[Corollary 2.4]{ji2023dao}, for every active $a_i$ (resp. $b_i$), $\mu_{f_V, a_i}$  (resp. $\mu_{g_V, b_i}$) and $\mu_{f_V,\bif}$ (resp. $\mu_{g_V,\bif}$) are proportional.   
Moreover  the proof of \cite[Corollary 2.4]{ji2023dao} implies that $\mu_{f_V,\bif}$ and $\mu_{g_V,\bif}$ are proportional.  Let $\mu$ be the probability measure on $V(\C)$ which is proportional to 
$\mu_{f_V,\bif}$ (hence $\mu_{g_V,\bif}$).

\medskip

Aside from the flexible Latt\`es locus, the exceptional maps \footnote{As in \cite[Section 1.1]{ji2023homoclinic}, we call $g$ \emph{exceptional} if it is a Latt\`es map or semiconjugates to a monomial map. } are isolated in $\sM_d(\C)$. As $\mu$ has continuous potential, we have the following property:
\begin{points}
\item[(1)] For $\mu$-a.e. point  $t\in V(\C)$, both $f_t$ and $g_t$ are non-exceptional. 
\end{points}

Let
$\Corr(\P^1)^{f_t\times g_t}_*$ be the set of $f_t\times g_t$-invariant Zariski closed subsets $\Gamma_t\subseteq \P^1\times \P^1$ of pure dimension $1$ such that both $\pi_1|_{\Gamma_t}$ and $\pi_2|_{\Gamma_t}$ are finite, where $\pi_1,\pi_2$ are the first and the second projections.
Let $\Corr^{\flat}(\P^1_{V})^{f_V\times_V g_V}_*$ be the set of $f_V\times_{V}g_V$-invariant Zariski closed subsets $\Gamma\subseteq V\times(\P^1\times \P^1)$ which are flat over $V$ and whose generic fiber is in $\Corr(\P^1_{\eta})^{f_\eta\times g_{\eta}}_*$, where $\eta$ is the generic point of $V$. In general, a correspondence $\Gamma_t\in \Corr(\P^1)^{f_t\times g_t}_*$ may not be contained in any correspondence in $\Corr^{\flat}(\P^1_{V})^{f_V\times_V g_V}_*$.
On the other hand, by the proof of \cite[Proposition 3.11]{ji2023dao}, this is the case if $t$ is \emph{transcendental} i.e. $t\in V(\C)\setminus V(\overline{\Q})$ . 
As $\mu$ has continuous potential, $\mu$-a.e. point  $t\in V(\C)$ is transcendental. 
We then get the following property:
\begin{points}
\item[(2)]For $\mu$-a.e. point  $t\in V(\C)$, every $\Gamma_t\in \Corr(\P^1)^{f_t\times g_t}_*$ is contained in some correspondence in $\Corr^{\flat}(\P^1_{V})^{f_V\times_V g_V}_*$.
\end{points}
Next, we consider some typical non-uniformly hyperbolic conditions. See \cite[Definition 4.1, 5.1 and 8.1]{ji2023dao} for the definitions of such conditions.
By \cite[Proposition 8.3]{ji2023dao}, there exists  $\la_0>1$ such that 
\begin{points}
\item[(3)] for $\mu-$a.e. point $t$, $f_t, g_t$ are $\CE(\la_0)$ hence $\TCE(\la_1)$ for every $1<\la_1<\la_0$, by Przytycki-Rohde \cite{przytycki1998porosity}.
\end{points}
By \cite[Theorem 4.3]{ji2023dao}, which is essentially due to De Th\'elin-Gauthier-Vigny \cite{de2021parametric}, 
\begin{points}
\item[(4)] for $\mu-$a.e. point $t$, $f_t, g_t$ are $\PCE(\la_2)$ for some $1<\la_2<d^{1/2}$.
\end{points}
By \cite[Theorem 4.6]{ji2023dao}, 
\begin{points}
\item[(5)] for $\mu-$a.e. point $t$, $f_t, g_t$ are $\PR(s)$ for some $s>0$.
\end{points}

 Let $a_1$ (resp. $b_1$)  be an active marked critical point of $f_V$ (resp $g_V$). As we have the conditions (1), (3), (4), (5), we may apply the proof of \cite[Proposition 7.8]{ji2023dao} for the active marked points $a_1,b_2$ to show that, for $\mu$-a.e. point $t\in V(\C)$, there is $\Gamma_t\in \Corr(\P^1)^{f_t\times g_t}_*$. By condition (2), $\Corr^{\flat}(\P^1_{V})^{f_V\times_V g_V}_*\neq\emptyset$. This implies that for all but finitely many $t\in V(\C)$, $f_t$ and $g_t$ are intertwined, which concludes the proof.
\endproof

\subsection{Proof of the generic injectivity} 

\proof[Proof of Theorem \ref{main}]
The case $d=2$ was proved by Milnor \cite{Milnor1993} (see also \cite[Theorem 2.45]{Silverman2012}).
The case $d=3$ was proved by Gotou \cite[Theorem 1.2]{Gotou2023} (ee also \cite{Hutz2013a} which is the  errata for \cite{Hutz2013}).

\medskip

Now we assume that $d\geq 4$. Assume by contradiction that $\tau_d$ is not generically injective.  By Theorem \ref{intertwined}, there is 
a non-empty Zariski open subset $U$ of $\sM_d^\star$ such that for every $f\in \Rat_d(\C)$ with $[f]\in U$, $f$  satisfies condition (i) and (ii) in Theorem \ref{intertwined}.
There is a Zariski open subset $W$ of the Zariski closure of $\tau_d(U)$ such that $\tau_d^{-1}(W)\subseteq U$ and $\tau_d|_{\tau_d^{-1}(W)}: \tau_d^{-1}(W)\to W$ if finite \'etale of degree at least $2$.
After shrinking $U$, we may assume that $U=\tau_d^{-1}(W).$ We fix this Zariski open subset $U$.

\medskip
For an algebraic family of rational maps $f$, we let $\Psi_f: V\to \sM_d$ be the morphism sending $t$ to the conjugacy class of $f_t$.  Now  we   construct two non-isotrivial algebraic families $f_V,g_V$  of degree $d$ rational maps parametrized by the same irreducible algebraic curve $V$ such that the following holds:
  There exists $\underline{n}:=\left\{n_1,\dots,n_{2d-3}\right\}\in (\N^*)^{2d-3}$ such that we have
\begin{equation}\label{31}
	\Psi_f(V)\subset Y_{\underline{n}}\cap U,\;\text{and} \;\Psi_g(V)\subset U,
\end{equation}\
where $Y_{\underline{n}}$ is the algebraic curve in Lemma \ref{curve}, and we have 
\begin{equation}\label{equequmul}
\tau_d\circ \Psi_f=\tau_d\circ \Psi_g,
\end{equation}
finally for every $t\in V(\C)$, we have
\begin{equation}\label{equnotconj} \Psi_f(t)\neq \Psi_g(t).
\end{equation}

We now describe this construction. By Lemma \ref{curve} (2), there exists $\underline{n}:=\left\{n_1,\dots,n_{2d-3}\right\}\in (\N^*)^{2d-3}$  such that $Y_{\underline{n}}\cap U\neq \emptyset.$
We pick an irreducible component $Z$ of $Y_{\underline{n}}$ meeting $U.$ By Lemma \ref{curve} (1), $Z$ is an algebraic curve. 
Set $Z_1:=Z\cap U.$ Then $Z'_1:=\tau_d(Z_1)$ is an algebraic  curve. Consider the fiber product $X:=Z_1\times_{Z'_1}(\tau_d|_U)^{-1}(Z'_1).$ Recall that by the construction of $U$, $\tau_d|_U: U\to W$ is finite \'etale of degree at least $2$. 
This implies that there is an irreducible component $\tilde{Z}$ of $X$ which is not the diagonal. 
We denote by $\pi_1:X\to Z_1$ and $\pi_2:X\to  (\tau_d|_U)^{-1}(Z'_1)$ for the first and the second projections. Both of $\pi_1$ and $\pi_2$ are finite \'etale.
Set $Z_2:=\pi_2(\tilde{Z}).$ For every $z\in \tilde{Z}$, $\pi_1(z)\neq \pi_2(z).$

Pick an irreducible curve $C_1'$ in $\Psi^{-1}(Z_1)\subseteq \Rat_d(\C)$, such that $\Psi|_{C_1'}: C_1'\to Z_1$ is dominant, where $\Psi:\Rat_d(\C)\to \sM_d(\C)$ is the quotient morphism.
Pick an irreducible component $C_1$ of $C_1'\times_{Z_1}\tilde{Z}$. Let $\phi_1,\phi_2$ be the projection to the first and the second coordinates.
We then get an algebraic family $h_{1}$ of degree $d$ rational maps parametrized by $C_1$ such that for every $t\in C_1$, the map $f_t$ is the one given by $\phi_1(t)\in \Rat_d(\C)$.
Then the morphism $\Psi_{1}:C_1\to \sM_d(\C)$ sending $t$ to the conjugacy class of $f_t$ is the composition of $\phi_2: C_1\to \tilde{Z}$ and $\pi_1: \tilde{Z}\to Z_1\subseteq \sM_d(\C)$.

Similarly, we get an algebraic family $h_2$ of degree $d$ rational maps parametrized by an irreducible curve $C_2$ and a dominant morphism $\phi_2': C_2\to Z_2$
such that $\Psi_{2}:C_2\to \sM_d(\C)$ sending $t$ to the conjugacy class of $f_t$ is the composition of $\phi_2': C_2\to \tilde{Z}$ and $\pi_2: \tilde{Z}\to Z_2\subseteq \sM_d(\C).$

Pick an irreducible component $V$ of $C_1\times_{\tilde{Z}} C_2$ and denote by $\beta_i, i=1,2$ the projections to the first and the second coordinates.
They induce two algebraic families $f_V,g_V$  of degree $d$ maps parametrized by the same irreducible algebraic curve $V$.
The image of  the morphism $\Psi_f: V\to \sM_d(\C)$ sending $t$ to the conjugacy class of $f_t$  is Zariski dense in $Z_1$ and the image of $\Psi_g: V\to \sM_d(\C)$ sending $t$ to the conjugacy class of $g_t$ is Zariski dense in $Z_2.$ Hence $f_V$ and $g_V$ are non-isotrivial. By our construction (\ref{31}) automatically holds. Moreover we have 
\begin{equation*}\tau_d\circ \Psi_f=\tau_d\circ \Psi_g.
\end{equation*}
As  for every $z\in \tilde{Z}$, $\pi_1(z)\neq \pi_2(z)$, for every $t\in V(\C)$, we have
\begin{equation*} \Psi_f(t)\neq \Psi_g(t).
\end{equation*}
Thus we complete the construction.

\medskip

%
%
We need the following lemma.
\begin{lem}\label{multiplier}
Let $d\geq 2$ and let $f,g\in\Rat_d(\C)$ have the same multiplier spectrum. If 
$f$ is a  hyperbolic PCF map of disjoint type  $\underline{n}:=\left\{n_1,\dots,n_{2d-2}\right\}\in (\N^*)^{2d-2}$, then $g$ is also a  hyperbolic PCF map of disjoint type  $\underline{n}$. 
\end{lem}
\begin{proof}
We will show that for every $f\in \Rat_d(\C)$, we can read the number of the critical cycles and length of them from the multiplier spectrum of $f$.

For $n\geq 1$, we view $S_n(f)$ as the multi-set $\{df^n(x)|\,\, x\in \Fix(f^n)\}$. It has exactly $d^n+1$ elements. 
We denote by $u_n$ the number of zeros in $S_n(f)$. It is clear that the sequence $u_n$ is determined by the multiplier spectrum of $f$.
For every $l\geq 1$, denote by $m_l$ the number of critical cycles of length $l$.
It is clear that the sequence $m_l, l\geq 1$ determines the number and the length of the critical cycles.
So we only need to read the sequence $m_l, l\geq 1$ from $u_n, n\geq 1.$
It is clear that $u_n=\sum_{l|n}m_l.$ By Mobius inversion formula, we get $m_n=\sum_{l|n} \mu(l)u_{n/l}$, where $\mu: \N^*\to \{-1,0,1\}$ is the M\"obius function. This concludes the proof.
%
%
%
%
%
\end{proof}

We continue the proof of Theorem \ref{main}. By Lemma \ref{curve} (3), there are infinitely many $t\in V(\C)$ such that $f_t$ is a hyperbolic PCF map of disjoint type. By  (\ref{equequmul}), $f_t$ and $g_t$ have the same multiplier spectrum, hence by Lemma \ref{multiplier}, for such $t$, $g_t$ is also a hyperbolic PCF map of disjoint type. By Theorem \ref{DAO}, after shrinking $V$, $f_t$ and $g_t$ are intertwined for every  $t\in V(\C)$.  Since (\ref{31}) holds and   $f_t$ and $g_t$ are intertwined for all but finitely many $t\in V(\C)$,  by Theorem \ref{intertwined}, $\Psi_f(t)=\Psi_g(t)$ for all but finitely many $t\in V(\C)$. This contradicts (\ref{equnotconj}). We then conclude the proof.
\endproof

\medskip

\proof[Proof of Theorem \ref{thmpoly}]

The proof of Theorem \ref{thmpoly} follows exactly the same line as the proof of  Theorem \ref{main}, with only one different point: In the proof of Theorem \ref{main}, a key ingredient is Theorem \ref{intertwined} for intertwined  rational maps  due to Pakovich, which works only when $d\geq 4.$ So we need to apply the results of Milnor \cite{Milnor1993} and Gotou \cite[Theorem 1.2]{Gotou2023} to treat the $d=2,3$ cases. To prove Theorem \ref{thmpoly}, we may replace Theorem \ref{intertwined} by \cite[Theorem 23 and 24]{favre2022arithmetic} for intertwined polynomials, which works for every $d\geq 2.$
\endproof
	
		\begin{rem}
	In the proof Theorem \ref{main},  instead of using Lemma \ref{multiplier}, we can use the following  fact weaker than Lemma \ref{multiplier}:  if two rational maps $f$ and $g$ have the same multiplier spectrum such that $f$ is PCF, then $g$ is also PCF.  This is a consequence of  \cite[Theorem 1.12]{Ji2023}.
		\end{rem}
		
\begin{rem}\label{remallpcf} There is also a way to prove Theorem \ref{main} considering all PCF maps instead of hyperbolic PCF maps of disjoint type.
In such a way, we may replace Gauthier-Okuyama-Vigny's result \cite[Theorem F]{gauthier2019hyperbolic} (c.f. Theorem \ref{GOV}) by the well-known fact that the PCF parameters are Zariski dense in $\sM_d$. On the other hand, Lemma \ref{multiplier} is not sufficient for the proof and we need to use \cite[Theorem 1.12]{Ji2023} instead. As \cite[Theorem 1.12]{Ji2023} relies on Siegel's deep theorem for integer points and the proof of Lemma \ref{multiplier} is elementary, we choose the current proof.
		\end{rem}
\subsection{An alternative approach}\label{subsectionsimp}
In a private communication, DeMarco and Mavraki told us that when they preparing lectures in Harvard and Toronto in November 2023  \cite{DeMarco2023}, they find a simplification of our proof of  Theorem \ref{main}. In particular, arithmetic equidistribution is not needed.  The aim of this section is to outline their  different approach.

DeMarco-Mavraki's  key observation is  as follows:
A key ingredient of Theorem \ref{main} is Theorem \ref{DAO} (=Theorem \ref{DAOint}) which proves more than what we need. 
The proof of Theorem \ref{DAO} strongly relies on the results and methods developed in author's 
previous paper \cite{ji2023dao}.

First, in the proof of Theorem \ref{main},  we do not need Theorem \ref{DAO} for general families. 
If we choose the way to prove Theorem \ref{main} as in Remark \ref{remallpcf}, we only need to apply Theorem \ref{DAO} for families $f_V,g_V$ having 
the following additional assumptions:
\begin{points}
\item[(A)] For every $t\in V(\C)$, the multiplier spectrums for $f_t$ and $g_t$ are the same. 
\item[(B)]  The curve $V$ is smooth and both $f_V$ and $g_V$ have exactly $2d-2$ marked critical points $a_1,\dots, a_{2d-2}$ and $b_1,\dots, b_{2d-2}$. Moreover all these marked critical points except $a_1,b_1$ 
 are strictly preperiodic at every $t\in V(\C)$.
\end{points}
Second, in Theorem \ref{DAO}, we proved that $f_t$ and $g_t$ are intertwined for all but finitely many $t\in V(\C)$. However, to  prove Theorem \ref{main}, we only need one such $t.$
So we may replace Theorem \ref{DAO} with  the following weaker result.
\begin{lem}\label{lemdmweak}
Let $f_V, g_V$ be families as in Theorem \ref{DAO}. We further assume $(A), (B)$ hold. Then there is $t_0\in V(\C)$ such that $f_{t_0}$ and $g_{t_0}$ are intertwined.
\end{lem}
The proof of Lemma \ref{lemdmweak} follows  the same strategy of  the proof of Theorem \ref{DAO}.
But there are two important steps become easier. 
After replacing $V$ by its normalization and then by some finite ramification cover, we may assume that $V$ is smooth and both $f_V$ and $g_V$ have exactly $2d-2$ marked critical points $a_1,\dots, a_{2d-2}$ and $b_1,\dots, b_{2d-2}.$

In the first step of the proof of Theorem \ref{DAO}, we show that there is a probability measure on $V(\C)$ which is proportional to $\mu_{f_V, a_i}$  (resp. $\mu_{g_V, b_i}$) for every active $a_i$ (resp. $b_i$).
The proof is based on Yuan-Zhang's deep equidistribution theorem \cite[Theorem 6.2.3]{yuan2021adelic} on quasi-projective varieties.
The corresponding step for Lemma \ref{lemdmweak} is much easier. By assumption (B), $a_1,b_1$ are the only active marked points for $f_V$ and $g_V$ respectively. Hence $\mu_{f_V, a_1}$ (resp. $\mu_{g_V, b_1}$) is proportional to $\mu_{f_V,\bif}$ (resp. $\mu_{g_V,\bif}$). By \cite{Berteloot2010} the Lyapunov exponent $L(h)$ for every rational function $h$ on $\P^1(\C)$ can be computed using its mutiplier spectrum. Then (A) implies that $L(f_t)=L(g_t)$ for every $t\in V(\C).$ By 
\cite{DeMarco2001,DeMarco2003},  the function $t\in V(\C)\mapsto L(f_t)$ (resp.  $t\in V(\C)\mapsto L(g_t)$) is a potential of 
$\mu_{f_V,\bif}$ (resp. $\mu_{g_V,\bif}$). This gives a logically easier proof of this step for Lemma \ref{lemdmweak}.

Another crucial step to prove Theorem \ref{DAO} is to construct similarities between the bifurcation measure on the parameter space $V(\C)$ and the maximal entropy measures on the phase spaces.  As the conclusion of Theorem \ref{DAO} is for all but finitely many parameters in $V(\C)$, we construct such similarities for $\mu$-a.e. $t\in V(\C)$.  
This step is essentially done in \cite{ji2023dao}, which is technically difficult. We first show that for a $\mu$-generic parameter $t$, $f_t$ and $g_t$ satisfy some Collet-Eckmann-type conditions (which can be thought as some weak hyperbolicity). We then built the similarities under such conditions. As our hyperbolicity assumption is very weak, we can not apply the previous methods as in \cite{Tan1990,favre2022arithmetic,gauthier2018dynamical} to construct similarities. Actually, our proof is based on subtle estimates of  distortions for non-injective maps and a suitable binding argument. 
This step becomes much easier for Lemma \ref{lemdmweak}, as we only need to construct the similarities at a single parameter $t_0$. 
By \cite[Theorem 0.1]{Dujardin2013}, there is a dense set $P$ of parameters $t\in \Supp\, \mu_{f_V,a_1}$ such that $a_1$ is transversally pre-repelling at
$t$. By (B), for every $t\in P$, $f_t$ is PCF with no periodic critical point. 
By \cite[Theorem 1.12]{Ji2023}, for every $t\in P$, $g_t$ is also PCF. Moreover,
the proof of Lemma \ref{multiplier} implies that  $g_t$ has no critical periodic point. In particular, $b_1(t)$ is pre-repelling for $g_t$.
Pick any point $t_0\in P$, it is easy to get the similarities at $t_0$ applying the method in \cite{favre2022arithmetic,gauthier2018dynamical}.

	\end{document}